\documentclass[12pt]{amsart}

\usepackage{amssymb,amsmath}
\usepackage{amssymb,latexsym}
\usepackage{amsfonts}
\usepackage{amssymb}
\usepackage{stmaryrd}
\usepackage{longtable}
\usepackage{amsmath, geometry, amssymb} 
\usepackage{chngcntr}
\usepackage{lscape}
\counterwithin{table}{section}
\numberwithin{equation}{section}
\newtheorem{theorem}{Theorem}[section]

\newtheorem{remark}{Remark}[section]

\newcommand{\sqr}[2]{{\vcenter{\vbox{\hrule height#2pt
                \hbox{\vrule width#2pt height#1pt \kern#1pt
                \vrule width#2pt}\hrule height#2pt}}}}

\newcommand{\mb}{\mbox}

\newcommand{\beq}{\begin{equation}}
\newcommand{\eeq}{\end{equation}}
\newcommand{\beqar}{\begin{eqnarray}}
\newcommand{\eeqar}{\end{eqnarray}}
\def\beqars{\begin{eqnarray*}}
\def\eeqars{\end{eqnarray*}}
\def\eop{\hfill\mb{$\hspace{12pt}\vrule height 7pt width 6pt depth 0pt$}}

\newcommand{\smod}[1]{\hspace{-1mm} \pmod{#1}}

\newcommand{\qu}[2]{\Bigl({\frac{#1}{#2}}\Bigr) }
\newcommand{\dqu}[2]{\ds{\qu{#1}{#2}}}

\def \ds{\displaystyle}

\newcommand{\nn}{\mathbb{N}}
\newcommand{\zz}{\mathbb{Z}}
\newcommand{\qq}{\mathbb{Q}}
\newcommand{\cc}{\mathbb{C}}

\newcommand{\hh}{\mathbb{H}}

\allowdisplaybreaks

\begin{document}

\title{Theta Products and Eta Quotients of Level $\mathbf{24}$ and Weight $\mathbf{2}$}

\author{Ay\c{s}e Alaca, \c{S}aban Alaca, Zafer Selcuk Aygin}

\maketitle

\markboth{AY\c{S}E ALACA, \c{S}ABAN ALACA, ZAFER SELCUK AYGIN}
{THETA PRODUCTS AND ETA QUOTIENTS OF LEVEL 24 AND WEIGHT 2} 

\begin{abstract}

We find bases for the spaces $M_2\Big(\Gamma_0(24),\dqu{d}{\cdot}\Big)$  ($d=1,8,12, 24$) of modular forms. 
We determine the Fourier coefficients of all $35$ theta products  $\varphi[a_1,a_2,a_3,a_4](z)$ in these spaces. 
We then deduce formulas for the number of representations of a positive integer $n$ by diagonal quaternary quadratic forms with coefficients $1$, $2$, $3$ or $6$   in a uniform manner, of which  $14$ are Ramanujan's universal quaternary quadratic forms. 
We also find all the eta quotients in the Eisenstein spaces $E_2\Big(\Gamma_0(24),\dqu{d}{\cdot}\Big)$  ($d=1,8,12, 24$) 
and give their Fourier coefficients.  \\ 


\end{abstract}

\vspace{2mm}

\noindent
\section{Introduction and Notation}

Let $\nn$, $\nn_0$, $\zz$, $\qq$ and $\cc$ denote the sets of positive integers, non-negative integers, integers, rational numbers and complex numbers, respectively. 
Let $N\in\nn$. Let $\Gamma_0(N)$ be the modular subgroup defined by
\beqars
\Gamma_0(N) = \left\{ \left(
\begin{array}{cc}
a & b \\
c & d
\end{array}
\right)  \big | \, a,b,c,d\in \zz ,~ ad-bc = 1,~c \equiv 0 \smod {N}
\right\} .
\eeqars 
We define a Dirichlet character $\chi_t$ for each $t\in\{-24, -8, -4, -3, 1, 8, 12, 24\}$ by
\beqar
 \chi_{t} ( m)=\dqu{t}{m} \text{ for }  m \in \zz .
\eeqar
Note that $\chi_1$ is the trivial character. 
Let $\chi$ and $\psi$ be Dirichlet characters. 
For $n\in \nn$ we define the generalized sums of divisors functions  $\ds \sigma_{(\chi,\psi )}(n)$ by
\beqar
\sigma_{(\chi,\psi )}(n) :=\sum_{1 \leq m|n}\chi(m)\psi(n/m)m.
\eeqar 
If $n \not\in \nn$ we set $\sigma_{(\chi,\psi )}(n)=0$. 
If $\chi=\psi=\chi_1$, then $\sigma_{(\chi,\psi)}(n) $ coincides with  the sum of divisors function 
\beqars
\sigma(n) = \sum_{1\leq m \mid n} m.
\eeqars

Let $k \in \zz$. We write $M_k(\Gamma_0(N),\chi)$ to denote the space of modular forms of weight $k$
with multiplier system $\chi$ for  $\Gamma_0(N)$, and 
$E_k (\Gamma_0(N),\chi)$ and $S_k(\Gamma_0(N),\chi)$ to denote the subspaces of Eisenstein forms and cusp forms 
of  $M_k(\Gamma_0(N),\chi)$, respectively. We also write $M_k(\Gamma_0(N))$, $E_k(\Gamma_0(N))$ and $S_k(\Gamma_0(N))$  for $M_k(\Gamma_0(N),\chi_1)$,  $E_k(\Gamma_0(N),\chi_1)$ and  $S_k(\Gamma_0(N),\chi_1)$, respectively.
It is known  (see \cite[p. 83]{stein} ) that
\beqar
M_k (\Gamma_0(N),\chi) = E_k (\Gamma_0(N),\chi) \oplus S_k(\Gamma_0(N),\chi).
\eeqar
The Dedekind eta function $\eta (z)$ is the holomorphic function defined on  the upper half plane $\hh=\{ z \in \cc \mid \mbox{\rm Im}(z) >0 \}$ by 
\beqar
\eta (z) = e^{\pi i z/12} \prod_{n=1}^{\infty} (1-e^{2\pi inz}).
\eeqar
Throughout the remainder of this paper we set $q=q(z):=e^{2\pi i z}$ with  $z\in \hh$.
Let $N \in \nn$ and $r_\delta \in \zz$ for each positive divisor $\delta$ of $N$. We define an eta quotient of  level $N$ by the product formula
\beqar
f(z) = \prod_{1\leq \delta \mid N} \eta^{r_{\delta}} ( \delta z) .
\eeqar 
We define an  Eisenstein series $E_{t_1,t_2} (q)$ by
\begin{align}
\ds E_{t_1,t_2} (z):=C_{t_1,t_2}+\sum_{n=1}^{\infty} \sigma_{(\chi_{t_1},\chi_{t_2})}(n)q^n
\end{align}
for each
\begin{align*}
(t_1,t_2)=&(-8,-3), (-3,-8), (1,24), (24,1), (1,1),(1,8), (8,1),\\
&(1,12), (12, 1), (-3,-4),(-4,-3),
\end{align*}
where
\begin{align*}
& C_{8,1}=-\frac{1}{2},~C_{12,1}=-1, ~C_{24,1}=-3,  ~C_{1,1}=-\frac{1}{24},\\
& C_{-8,-3}=C_{-3,-8}=C_ {1,24}=C_{1,8}=C_{1,12}=C_{-3,-4}=C_{-4,-3}=0.
\end{align*}
For $(t_1, t_2)=(1,1)$ we write
\beqars
L(q):=E_{1,1}(z)=-\frac{1}{24}+\sum_{n=1}^{\infty} \sigma(n) q^n.
\eeqars
For $1<d \mid 24 $ we set
\beqar
L_d(q):=L(q)-dL(q^d) . 
\eeqar
Ramanujan's theta function $\varphi (z)$ is defined by
\beqars
\varphi(z) = \sum_{n=-\infty}^\infty q^{  n^2 }. 
\eeqars
It is well known  \cite[p. 11]{Berndt} that $\varphi(z)$ can be expressed as
\beqar
\varphi(z)=\frac{\eta^5(2z)}{\eta^2(z) \eta^2(4z)}.
\eeqar
Let $a_1,~a_2,~a_3,~a_4 \in \nn$ and $n\in \nn_0$. We define 
\beqars
N(a_1,a_2, a_3, a_4;n):=\mbox{card}\{(x_1,  x_2, x_3, x_4) \in \zz^{4} \vert  \, n=a_1 x_1^2 +a_{2} x_{2}^2+a_{3} x_{3}^2  +a_{4} x_{4}^2  \} .
\eeqars
For notational convenience we set
\beqars
\varphi[a_1,a_2,a_3,a_4](z):=\varphi(a_1 z) \varphi(a_2 z)\varphi(a_3 z)\varphi(a_4 z).
\eeqars
We have
\beqar
\varphi[a_1,a_2,a_3,a_4](z)=\sum_{n=0}^\infty N(a_1,a_2, a_3, a_4;n) q^{ n  },
\eeqar
which is independent of the order of $a_1,a_2,a_3,a_4$.
We have $35$  theta products $\varphi[a_1,a_2,a_3,a_4](z)$ of level $24$ and weight $2$.
We group them according to the modular spaces to which they belong, namely 
\begin{align}
\left\{ 
\begin{array}{l l l l}
\varphi[1,1,1,1](z), & \varphi[1,1,2,2](z), & \varphi[1,1,3,3](z), &\varphi[1,1,6,6](z), \\
 \varphi[1,2,3,6](z), & \varphi[2,2,2,2](z),&\varphi[2,2,3,3](z), & \varphi[2,2,6,6](z),\\
  \varphi[3,3,3,3](z),& \varphi[3,3,6,6](z), & \varphi[6,6,6,6](z); &
\end{array}
\right.
\end{align}
\begin{align}
\left\{ 
\begin{array}{l l l l}
\varphi[1,1,1,2](z), & \varphi[1,1,3,6](z), & \varphi[1,2,2,2](z), & \varphi[1,2,3,3](z), \\
\varphi[1,2,6,6](z), & \varphi[2,2,3,6](z),  &\varphi[3,3,3,6](z), & \varphi[3,6,6,6](z) ;
\end{array}
\right. 
\end{align}
\begin{align}
\left\{ 
\begin{array}{l l l l}
 \varphi[1,1,1,3](z), & \varphi[1,1,2,6](z), & \varphi[1,2,2,3](z),& \varphi[1,3,3,3](z),\\
 \varphi[1,3,6,6](z), & \varphi[2,2,2,6](z), &\varphi[2,3,3,6](z), & \varphi[2,6,6,6](z) ;
\end{array}
\right.
\end{align}
\begin{align}
\left\{ 
\begin{array}{l l l l}
\varphi[1,1,1,6](z), & \varphi[1,1,2,3](z), & \varphi[1,2,2,6](z), &\varphi[1,3,3,6](z),\\
\varphi[1,6,6,6](z), & \varphi[2,2,2,3](z), &\varphi[2,3,3,3](z), & \varphi[2,3,6,6](z)  
\end{array}
\right. 
\end{align}
are in $M_2(\Gamma_0(24),\chi_1)$, $ M_2(\Gamma_0(24),\chi_{8})$,  $ M_2(\Gamma_0(24),\chi_{12})$ and $M_2(\Gamma_0(24),\chi_{24})$, respectively.

In this paper we give bases for the spaces $M_2\Big(\Gamma_0(24),\dqu{d}{\cdot}\Big)$  ($d=1,8,12, 24$) of modular forms. 
We determine the Fourier cofficients of all $35$ theta products in (1.10)--(1.13). 
We then deduce explicit formulas for $N(a_1, a_2, a_3, a_4 ;n)$, where $a_1,a_2,a_3,a_4 \in \{ 1,2,3,6 \}$,   in a uniform manner, 
of which $14$ are Ramanujan's universal quaternary quadratic forms given in \cite{ramanujan}.  
We also find all the eta quotients in the Eisenstein spaces $E_2\Big(\Gamma_0(24),\dqu{d}{\cdot}\Big)$  ($d=1,8,12, 24$) 
and give their Fourier coefficients. 

\section{Bases for $M_2(\Gamma_0(24),\chi_i)$ for $ i\in \{1, 8, 12, 24\}$}

We deduce from \cite[Sec. 6.1, p. 93]{stein} that
\begin{eqnarray}
\dim(S_2(\Gamma_0(24))=1, ~ \dim(E_2(\Gamma_0(40))=7.
\end{eqnarray}
We also deduce from \cite[Sec. 6.3, p. 98]{stein} that
\begin{eqnarray}
&&\dim(S_2(\Gamma_0(24),\chi_8)=2, 
\dim(S_2(\Gamma_0(24),\chi_{12})=0, 
\dim(S_2(\Gamma_0(24),\chi_{24})=2, 
\end{eqnarray}
and
\begin{eqnarray}
&&\dim(E_2(\Gamma_0(24),\chi_8)=4,
\dim(E_2(\Gamma_0(24),\chi_{12})=8,
\dim(E_2(\Gamma_0(24),\chi_{24})=4.
\end{eqnarray}
We define the eta quotients 
\beqar
&&A(q):= \eta(2z)\eta(4z)\eta(6z)\eta(12z),\\
&&B_1(q):=\frac{\eta(z)\eta^4(6z)\eta^2(8z)}{\eta(2z)\eta(3z)\eta(12z)},~\hspace{10mm}
B_2(q):=\frac{\eta^2(z)\eta(8z)\eta^4(12z)}{\eta(4z)\eta(6z)\eta(24z)},\\
&&C_1(q):=\frac{\eta(z)\eta(4z)\eta^4(6z)\eta^2(24z)}{\eta(2z)\eta(3z)\eta^2(12z)},~
C_2(q):=\frac{\eta^2(z)\eta^4(4z)\eta(6z)\eta(24z)}{\eta^2(2z)\eta(8z)\eta(12z)}.
\eeqar

We now give a basis for $M_2(\Gamma_0(24),\chi_i)$ for each $ i\in \{1, 8, 12, 24\}$.
\begin{theorem} Let $\chi_t$ be as in {\em(1.1)} for $t=1,8,12,24$. Then
\begin{align*}
&\begin{aligned}
\{ L_d(q) \mid d=2,3,4,6,8,12,24\}\cup \{A(q) \},
\end{aligned}\\
&\begin{aligned}
\left\{E_{1, 8} (z),~E_{1, 8} (3z),~E_{8,1} (z),~E_{ 8,1} (3z),~B_1(q),~B_2(q)\right\},
\end{aligned}\\
&\begin{aligned}
\left\{E_{1, 12} (z), E_{1, 12} (2z), E_{ 12,1} (z), E_{ 12,1} (2z), E_{-4,-3} (z), \right. \\
\left.  E_{-4,-3} (2z), E_{-3,-4} (z), E_{-3,-4} (2z)\right\}, 
\end{aligned}\\
&\begin{aligned}
\left\{E_{1, 24} (z),~E_{24,1} (z),~E_{ -3,-8} (z), ~E_{ -8,-3} (z),~C_1(q),~C_2(q)\right\},
\end{aligned}
\end{align*}
are bases for $M_2(\Gamma_0(24))$,  $M_2(\Gamma_0(24),\chi_8)$, $M_2(\Gamma_0(24),\chi_{12})$ and  $M_2(\Gamma_0(24),\chi_{24})$, respectively.
\end{theorem}

\begin{proof} Appealing to \cite[Theorem 1.64, p. 18]{onoweb} and  \cite[Corollary 2.3, p. 37]{Kohler} (see also  \cite{Ligozat, Kilford, alacaaygin-1}) 
one can show that $A(q)\in S_2(\Gamma_0(24))$; $B_1(q), B_2(q) \in  S_2(\Gamma_0(24),\chi_8)$; $C_1(q), C_2(q) \in  S_2(\Gamma_0(24),\chi_{24})$.
The assertion (i) follows from  (1.3), (2.1) and \cite[Theorem 5.8, p. 88]{stein}. (ii)  follows from (1.3), (2.2), (2.3) and  
\cite[Theorem 5.9, p. 88]{stein} with $\epsilon=\chi_8$ and $\chi, \psi \in \{ \chi_1, \chi_8\}$.
(iii) follows from (1.3), (2.2) and (2.3) \cite[Theorem 5.9, p. 88]{stein} with $\epsilon=\chi_{12}$ and $\chi, \psi \in \{\chi_1, \chi_{12}, \chi_{-3}, \chi_{-4} \}$.
(iv) follows from  (1.3), (2.2), (2.3) and \cite[Theorem 5.9, p. 88]{stein} with $\epsilon=\chi_{24}$ and $\chi, \psi \in \{\chi_1, \chi_{24}, \chi_{-3}, \chi_{-8}\}$.
\end{proof}

\section{Theta products in  $M_2(\Gamma_0(24),\chi_i)$ for  $ i\in \{1, 8, 12, 24\}$}

Theorems 3.1--3.4 below follow from  (1.10)--(1.13) and Theorem 2.1.

\begin{theorem} 
Let  $\varphi[a_1,a_2,a_3,a_4](z)\in M_2(\Gamma_0(24))$ be any of the theta products given in {\em(1.10)}, and let $L_d(q)$ be as in {\em (1.7)}. Then we have 
\beqars
\varphi[a_1,a_2,a_3,a_4](z) &=&  b_2 L_2(q)+b_3L_3(q)+b_4L_4(q)+b_6L_6(q) +b_8(L_8(q)\\
&&+b_{12}L_{12}(q)+b_{24}L_{24}(q) + x A(q),
\eeqars
where  the coefficients $b_2$,  $b_{3}$, $b_4$, $b_6$, $b_8$, $b_{12}$, $b_{24}$ and $x$ are given at the right hand side of {\em Table 3.1}.
\end{theorem}

{\small
\begin{center}  
\begin{longtable}{r r r r| r r r r r r r r}
\caption{\footnotesize $\varphi[a_1,a_2,a_3,a_4](z)= b_2 L_2(q)+b_3L_3(q)+b_4L_4(q)+b_6L_6(q) +b_8(L_8(q)+b_{12}L_{12}(q)+b_{24}L_{24}(q) + x A(q)$. } \\
\hline\hline 
$a_1$ & $a_2$ & $a_3$ & $a_4$ & $b_2$ & $b_{3}$ & $b_4$ & $b_6$ & $b_8$ & $b_{12}$ & $b_{24}$ & $x$\\ 
\hline
\endfirsthead
\hline
\hline 
 $a_1$ & $a_2$ & $a_3$ & $a_4$ & $b_1$ & $b_{2}$ & $b_3$ & $b_4$ & $b_5$ & $b_6$ & $b_7$ & $b_8$\\ 
\hline
\endhead
\hline 
\endfoot
\hline
\endlastfoot
$1$ & $ 1$ & $ 1$ & $1$ & $ 0$ & $0$ & $8$ & $0$ & $0$ & $0$ & $0$ & $0$\\
$1$ & $1$ & $ 2$ & $ 2$ & $ 2$ & $0$ & $-2$ & $0$ & $4$ & $0$ & $0$ & $0$\\
$1$ & $ 1$ & $ 3$ & $ 3$ & $ 4$ & $4$ & $-4$ & $-4$ & $0$ & $4$ & $0$ & $0$\\
$1$ & $ 1$ & $ 6$ & $ 6$ & $ 1$ & $2$ & $1$ & $-1$ & $-2$ & $-1$ & $2$ & $2$\\
$1$ & $ 2$ & $ 3$ & $ 6$ & $ 1/2$ & $-1$ & $-1/2$ & $1/2$ & $1$ & $-1/2$ & $1$ & $1$\\
$2$ & $ 2$ & $ 2$ & $2$ & $ -4$ & $0$ & $0$ & $0$ & $4$ & $0$ & $0$ & $0$\\
$2$ & $ 2$ & $ 3$ & $ 3$ & $ 1$ & $2$ & $1$ & $-1$ & $-2$ & $-1$ & $2$ & $-2$\\
$2$ & $ 2$ & $ 6$ & $6$ & $ -2$ & $0$ & $2$ & $2$ & $-2$ & $-2$ & $2$ & $0$\\
$3$ & $ 3$ & $ 3$ & $ 3$ & $ 0$ & $-8/3$ & $0$ & $0$ & $0$ & $8/3$ & $0$ & $0$\\
$3$ & $ 3$ & $ 6$ & $ 6$ & $ 0$ & $-4/3$ & $0$ & $2/3$ & $0$ & $-2/3$ & $4/3$ & $0$\\
$6$ & $ 6$ & $ 6$ & $ 6$ & $ 0$ & $0$ & $0$ & $-4/3$ & $0$ & $0$ & $4/3$ & $0$
\end{longtable}
\end{center}
}

{\bf Proof.} Let  $\varphi[a_1,a_2,a_3,a_4](z)$ be any of the theta products listed in  (1.10).
By Theorem 2.1,  $\varphi[a_1, a_2, a_3,a_4](z)$ must be a linear combination of $L_d(q)$ ($d=2,3,4,6,8,12,24$) 
and $A(q)$, namely 
\begin{align}
\varphi[a_1, a_2, a_3,a_4](z) =&\,   b_2 L_2(q)+b_3L_3(q)+b_4L_4(q)+b_6L_6(q) +b_8(L_8(q)\\
&+b_{12}L_{12}(q)+b_{24}L_{24}(q) + x A(q) \nonumber 
\end{align}
for some scalars $b_2,b_3,b_4,b_6,b_8,b_{12},b_{24}, x\in \cc$.
By \cite[Theorem 3.13]{Kilford}, the Sturm bound for the modular space $M_2(\Gamma_0(24))$ is $8$.
So, equating the coefficients of $q^{n}$ for $0\leq n\leq 8$ on both sides of (3.1),
we find a  system of linear equations.
We solve this system and find the asserted coefficients.
\eop

\vspace{2mm}

Similarly to Theorem 3.1, Theorems 3.2--3.4 follow from (1.11)--(1.13) and Theorem 2.1.

\begin{theorem} 
Let  $\varphi[a_1,a_2,a_3,a_4](z)\in M_2(\Gamma_0(24),\chi_{8})$ be any of the theta products given in {\em(1.11)}, 
where $\chi_{8}(n)$ is given by {\em (1.1)}. Let $E_{1,8}(z)$ and $E_{8,1}(z)$ be as in {\em (1.6)}.
Then we have 
\beqars
\varphi[a_1,a_2,a_3,a_4](z) &=&  b_1 E_{1, 8} (z) + b_2 E_{1, 8} (3z) + b_3 E_{8,1} (z) +b_4 E_{ 8,1} (3z)\\
&&+ x_1 B_1(q)+x_2 B_2(q),
\eeqars
where the  coefficients $b_1$,  $b_{2}$, $b_3$, $b_4$,  $x_1$ and $x_2$ are given at the right hand side of {\em Table 3.2}.
\end{theorem}

{\small 
\begin{center} 
\begin{longtable}{r r r r| r r r r r r }
\caption{\footnotesize  $\varphi[a_1,a_2,a_3,a_4](z)= b_1 E_{1, 8} (z) + b_2 E_{1, 8} (3z) + b_3 E_{8,1} (z) +b_4 E_{ 8,1} (3z)+ x_1 B_1(q)+x_2 B_2(q) $} \\
\hline\hline 
$a_1$ & $a_2$ & $a_3$ & $a_4$ & $b_1$ & $b_{2}$ & $b_3$ & $b_4$ & $x_1$ & $x_2$ \\ 
\hline
\endfirsthead
\hline
\hline 
 $a_1$ & $a_2$ & $a_3$ & $a_4$ & $b_1$ & $b_{2}$ & $b_3$ & $b_4$ & $x_1$ & $x_2$ \\ 
\hline
\endhead
\hline 
\endfoot
\hline
\endlastfoot
$1$ & $ 1$ & $ 1$ & $ 2$ & $ 8$ & $ 0$ & $ -2$ & $ 0$ & $ 0$ & $ 0$\\
$1$ & $ 1$ & $ 3$ & $ 6$ & $ 16/5$ & $ -24/5$ & $ -4/5$ & $ -6/5$ & $ 8/5$ & $ 0$ \\
$1$ & $ 2$ & $ 2$ & $ 2$ & $ 4$ & $ 0$ & $ -2$ & $ 0$ & $ 0$ & $ 0$ \\
$1$ & $ 2$ & $ 3$ & $ 3$ & $ 8/5$ & $ 48/5$ & $ 2/5$ & $ -12/5$ & $ -8/5$ & $ 8/5$ \\
$1$ & $ 2$ & $ 6$ & $ 6$ & $ 4/5$ & $ 24/5$ & $ 2/5$ & $ -12/5$ & $ 8/5$ & $ -4/5$ \\
$2$ & $ 2$ & $ 3$ & $ 6$ & $ 8/5$ & $ -12/5$ & $ -4/5$ & $ -6/5$ & $ 0$ & $ -4/5$ \\
$3$ & $ 3$ & $ 3$ & $ 6$ & $ 0$ & $ 8$ & $ 0$ & $ -2$ & $ 0$ & $ 0$ \\
$3$ & $ 6$ & $ 6$ & $ 6$ & $ 0$ & $ 4$ & $ 0$ & $ -2$ & $ 0$ & $ 0$ 
\end{longtable}
\end{center}
}

\begin{theorem} 
Let  $\varphi[a_1,a_2,a_3,a_4](z)\in M_2(\Gamma_0(24),\chi_{12})$ be any of the theta products given in {\em(1.12)}, where $\chi_{12}(n)$ is given by {\em (1.1)}. 
Let $E_{1,12}(z)$, $E_{12,1}(z)$, $E_{-3,-4}(z)$ and $E_{-4,-3}(z)$ be as in {\em (1.6)}.
Then we have 
\begin{align*}
\varphi[a_1,a_2,a_3,a_4](z) =& b_1 E_{12, 1} (z) +b_2 E_{12,1} (2z) +b_3 E_{ 1,12} (z) +b_4 E_{ 1,12} (2z) \\
&+ b_5 E_{-4, -3} (z) +b_6 E_{-4,-3} (2z) +b_7 E_{ -3,-4} (z) +b_8 E_{ -3,-4} (2z),
\end{align*}
where the  coefficients $b_1$,  $b_{2}$, $b_3$, $b_4$,  $b_5$, $b_6$, $b_7$ and $b_8$ are given at the right hand side of  {\em Table 3.3}.
\end{theorem}

{\small 
\begin{center} 
\begin{longtable}{r r r r| r r r r r r r r }
\caption{\footnotesize $\varphi[a_1,a_2,a_3,a_4](z)=b_1 E_{12, 1} (z) +b_2 E_{12,1} (2z) +b_3 E_{ 1,12} (z) +b_4 E_{ 1,12} (2z)+ b_5 E_{-4, -3} (z) +b_6 E_{-4,-3} (2z) +b_7 E_{ -3,-4} (z) +b_8 E_{ -3,-4} (2z)$. } \\
\hline\hline 
$a_1$ & $a_2$ & $a_3$ & $a_4$ & $b_1$ & $b_{2}$ & $b_3$ & $b_4$ & $b_5$ & $b_6$ & $b_7$ & $b_8$ \\ 
\hline
\endfirsthead
\hline
\hline 
 $a_1$ & $a_2$ & $a_3$ & $a_4$ & $b_1$ & $b_{2}$ & $b_3$ & $b_4$ & $b_5$ & $b_6$ & $b_7$ & $b_8$ \\ 
\hline
\endhead
\hline 
\endfoot
\hline
\endlastfoot
$1$ & $ 1$ & $ 1$ & $ 3$ & $   -1$ & $  0$ & $  6$ & $  0$ & $  3$ & $  0$ & $  -2$ & $  0$ \\
$1$ & $ 1$ & $ 2$ & $ 6$ & $   0$ & $  -1$ & $  3$ & $  0$ & $  0$ & $  3$ & $  1$ & $  0$ \\ 
$1$ & $ 2$ & $ 2$ & $ 3$ & $   0$ & $  -1$ & $  3$ & $  0$ & $  0$ & $  -3$ & $  -1$ & $  0$ \\
$1$ & $ 3$ & $ 3$ & $ 3$ & $   -1$ & $  0$ & $  2$ & $  0$ & $  -1$ & $  0$ & $  2$ & $  0$ \\ 
$1$ & $ 3$ & $ 6$ & $ 6$ & $   0$ & $  -1$ & $  1$ & $  0$ & $  0$ & $  1$ & $  1$ & $  0$ \\ 
$2$ & $ 2$ & $ 2$ & $ 6$ & $   0$ & $  -1$ & $  0$ & $  6$ & $  0$ & $  3$ & $  0$ & $  -2$ \\ 
$2$ & $ 3$ & $ 3$ & $ 6$ & $   0$ & $  -1$ & $  1$ & $  0$ & $  0$ & $  -1$ & $  -1$ & $  0$ \\ 
$2$ & $ 6$ & $ 6$ & $ 6$ & $   0$ & $  -1$ & $  0$ & $  2$ & $  0$ & $  -1$ & $  0$ & $  2$ 
\end{longtable}
\end{center}
}

\begin{theorem} 
Let  $\varphi[a_1,a_2,a_3,a_4](z)\in M_2(\Gamma_0(24),\chi_{24})$ be any of the theta products given in {\em(1.13)}, where $\chi_{24}(n)$ is given by {\em (1.1)}.
Let $E_{1,24}(z)$, $E_{24,1}(z)$, $E_{-8,-3}(z)$ and $E_{-3,-8}(z)$ be as in {\em (1.6)}.
Then we have 
\begin{align*}
\varphi[a_1,a_2,a_3,a_4](z) =& b_1 E_{24, 1} (z)+b_2 E_{1,24} (z)+b_3 E_{ -8,-3} (z)+ b_4 E_{ -3,-8} (z) \\
&+x_1C_1(q) +x_2 C_2(q),
\end{align*}
where the coefficients $b_1$,  $b_{2}$, $b_3$, $b_4$, $x_1$ and $x_2$ are given at the right hand side of {\em Table 3.4}.
\end{theorem}

{\small 
\begin{center} 
\begin{longtable}{r r r r| r r r r r r }
\caption{\footnotesize $\varphi[a_1,a_2,a_3,a_4](z)=b_1 E_{24, 1} (z)+b_2 E_{1,24} (z)+b_3 E_{ -8,-3} (z)+ b_4 E_{ -3,-8} (z) +x_1C_1(q) +x_2 C_2(q)$.} \\
\hline\hline 
$a_1$ & $a_2$ & $a_3$ & $a_4$ & $b_1$ & $b_{2}$ & $b_3$ & $b_4$ & $x_1$ & $x_2$ \\ 
\hline
\endfirsthead
\hline
\hline 
 $a_1$ & $a_2$ & $a_3$ & $a_4$ & $b_1$ & $b_{2}$ & $b_3$ & $b_4$ & $x_1$ & $x_2$  \\ 
\hline
\endhead
\hline 
\endfoot
\hline
\endlastfoot
$1$ & $ 1$ & $ 1$ & $ 6$ & $  -1/3$ & $  4$ & $  1$ & $  -4/3$ & $  8$ & $  8/3$\\
$1$ & $ 1$ & $ 2$ & $ 3$ & $  -1/3$ & $  4$ & $  -1$ & $  4/3$ & $  0$ & $  0$ \\
$1$ & $ 2$ & $ 2$ & $ 6$ & $  -1/3$ & $  2$ & $  1$ & $  -2/3$ & $  0$ & $  0$ \\
$1$ & $ 3$ & $ 3$ & $ 6$ & $  -1/3$ & $  4/3$ & $  -1/3$ & $  4/3$ & $  0$ & $  0$ \\
$1$ & $ 6$ & $ 6$ & $ 6$ & $  -1/3$ & $  2/3$ & $  -1/3$ & $  2/3$ & $  8/3$ & $  4/3$ \\
$2$ & $ 2$ & $ 2$ & $ 3$ & $  -1/3$ & $  2$ & $  -1$ & $  2/3$ & $  0$ & $  -4/3$ \\
$2$ & $ 3$ & $ 3$ & $ 3$ & $  -1/3$ & $  4/3$ & $  1/3$ & $  -4/3$ & $  -8/3$ & $  0$ \\
$2$ & $ 3$ & $ 6$ & $ 6$ & $  -1/3$ & $  2/3$ & $  1/3$ & $  -2/3$ & $  0$ & $  0$ 
\end{longtable}
\end{center}
}

\begin{remark}
We obtain the following identities from {\rm Tables 3.1, 3.2} and {\rm 3.4}.  
\begin{align*}
&-\varphi[1,1,2,2](z) + 4 \varphi[1,2,3,6](z) -3\varphi[3,3,6,6](z)  = 4A(q), \\
& -2\varphi[1,1,1,2](z) + 5 \varphi[1,1,3,6](z) + 9 \varphi[3,3,3,6](z) -12\varphi[3,6,6,6](z)=8B_1(q) ,\\
&2\varphi[1,2,2,2](z)-5\varphi[2,2,3,6](z) -6\varphi[3,3,3,6](z) + 9 \varphi[3,6,6,6](z)=4B_2(q), \\
&2\varphi[1,1,1,6](z)-3 \varphi[1,1,2,3](z) +4 \varphi[2,2,2, 3](z) -3\varphi[2,3,3,3](z)=24C_1(q),\\ 
&2\varphi[1,1,2,3](z)-2\varphi[1,2,2,6](z)-3\varphi[2,2,2,3](z) +3\varphi[2,3,6,6](z)=4C_2(q). 
\end{align*}
\end{remark}

\section{Representations by quaternary quadratic forms with coefficients $1$, $2$, $3$ or $6$}

Let  $a_1, a_2, a_3, a_4\in\{1,2,3,6\}$. With the simplifying  assumptions 
\begin{align*}
\gcd(a_1,a_2,a_3,a_4)=1 \text{ and } a_1\leq a_2\leq a_3\leq a_4,
\end{align*}
there are $26$ diagonal quaternary quadratic forms $a_1 x_1^2+a_2 x_2^2+a_3 x_2^2+a_4 x_4^2$. 
Of these,  $14$ are Ramanujan's universal quaternary quadratic forms given in \cite{ramanujan}.
In Theorem 4.1 we give explicit formulas for $N(a_1,a_2,a_3,a_4;n)$ for these $14$ universal quaternary quadratic forms. 
In Theorem 4.2 we give formulas  for $N(a_1,a_2,a_3,a_4;n)$ for the remaining $12$ quaternary quadratic forms. Both Theorems 4.1 and 4.2 follow from Theorems 3.1--3.4.

\begin{theorem} 
Let $n \in \nn$.  Then
\begin{align*}
&\begin{aligned}
{\bf (i)~} N(1,1,1,1;n) = & 8\sigma(n) - 32 \sigma (n/4),
\end{aligned}\\
&\begin{aligned}
{\bf (ii)~}
 N(1,1,2,2;n) = & 4\sigma(n) - 4\sigma (n/2) +8 \sigma (n/4)  - 32\sigma(n/8),  
\end{aligned}\\
&\begin{aligned}
{\bf (iii)~}
 N(1,1,3,3;n) = &4 \sigma(n) -8 \sigma (n/2) -12 \sigma (n/3) +16 \sigma (n/4) \\
 &+ 24 \sigma(n/6)  -48\sigma(n/12), 
\end{aligned}\\
&\begin{aligned}
{\bf (iv)~} N(1,2,3,6;n) = &\sigma(n) - \sigma (n/2) + 3 \sigma (n/3) + 2 \sigma (n/4) -3\sigma(n/6) \\
 &- 8\sigma(n/8)  +6\sigma(n/12)  - 24\sigma(n/24) + a(n),
\end{aligned}\\
&\begin{aligned}
{\bf (v)~} N(1,1,1,2;n) = & 8\sigma_{(\chi_1,\chi_8)}(n) - 2\sigma_{(\chi_8,\chi_1)}(n) ,
\end{aligned}\\
&\begin{aligned}
{\bf (vi)~} N(1,1,3,6;n) = &\frac{16}{5} \sigma_{(\chi_1,\chi_8)}(n) - \frac{24}{5}\sigma_{(\chi_1,\chi_8)}(n/3) -\frac{4}{5}\sigma_{(\chi_8,\chi_1)}(n) \\
 &-\frac{6}{5}\sigma_{(\chi_8,\chi_1)}(n/3)+\frac{8}{5}b_1(n),
\end{aligned}\\
&\begin{aligned}
{\bf (vii)~}
 N(1,2,2,2;n) = & 4\sigma_{(\chi_1,\chi_8)}(n) -2\sigma_{(\chi_8,\chi_1)}(n),  
\end{aligned}\\
&\begin{aligned}
{\bf (viii)~}
 N(1,2,3,3;n) = &\frac{8}{5} \sigma_{(\chi_1,\chi_8)}(n) + \frac{48}{5}\sigma_{(\chi_1,\chi_8)}(n/3) +\frac{2}{5}\sigma_{(\chi_8,\chi_1)}(n) \\
 &-\frac{12}{5}\sigma_{(\chi_8,\chi_1)}(n/3)-\frac{8}{5}b_1(n)+\frac{8}{5}b_2(n), 
\end{aligned}\\
&\begin{aligned}
{\bf (ix)~} N(1,1,1,3;n) = & 6\sigma_{(\chi_1,\chi_{12})}(n) -\sigma_{(\chi_{12},\chi_1)}(n) -2\sigma_{(\chi_{-3},\chi_{-4})}(n) +3\sigma_{(\chi_{-4},\chi_{-3})}(n) ,
\end{aligned}\\
&\begin{aligned}
{\bf (x)~}
 N(1,1,2,6;n) = & 3\sigma_{(\chi_1,\chi_{12})}(n) -\sigma_{(\chi_{12},\chi_1)}(n/2) + \sigma_{(\chi_{-3},\chi_{-4})}(n) + 3\sigma_{(\chi_{-4},\chi_{-3})}(n/2), 
\end{aligned}\\
&\begin{aligned}
{\bf (xi)~}N(1,2,2,3;n) = &  3\sigma_{(\chi_1,\chi_{12})}(n) -\sigma_{(\chi_{12},\chi_1)}(n/2) -\sigma_{(\chi_{-3},\chi_{-4})}(n) -3\sigma_{(\chi_{-4},\chi_{-3})}(n/2) ,
\end{aligned}\\
&\begin{aligned}
{\bf (xii)~} N(1,1,1,6;n) = & 4\sigma_{(\chi_1,\chi_{24})}(n) -\frac{1}{3}\sigma_{(\chi_{24},\chi_1)}(n) -\frac{4}{3}\sigma_{(\chi_{-3},\chi_{-8})}(n) \\
&+\sigma_{(\chi_{-8},\chi_{-3})}(n) +8c_1(n) +\frac{8}{3}c_2(n),
\end{aligned}\\
&\begin{aligned}
{\bf (xiii)~}
 N(1,1,2,3;n) = &  4\sigma_{(\chi_1,\chi_{24})}(n) -\frac{1}{3}\sigma_{(\chi_{24},\chi_1)}(n) +\frac{4}{3}\sigma_{(\chi_{-3},\chi_{-8})}(n) -\sigma_{(\chi_{-8},\chi_{-3})}(n) ,
\end{aligned}\\
&\begin{aligned}
{\bf (xiv)~}
 N(1,2,2,6;n) = & 2\sigma_{(\chi_1,\chi_{24})}(n) -\frac{1}{3}\sigma_{(\chi_{24},\chi_1)}(n) -\frac{2}{3}\sigma_{(\chi_{-3},\chi_{-8})}(n) +\sigma_{(\chi_{-8},\chi_{-3})}(n) .
\end{aligned}
\end{align*}
\end{theorem}

\begin{theorem}  
Let $n \in \nn$.  Then
\begin{align*}
&\begin{aligned}
{\bf (i)~}
N(1,1,6,6;n) = & 2 \sigma(n) -2 \sigma (n/2) -6 \sigma (n/3) -4 \sigma (n/4)+ 6 \sigma(n/6)  \\
 &+ 16 \sigma(n/8) +12\sigma(n/12) - 48 \sigma(n/24) +2 a(n),
\end{aligned}\\
&\begin{aligned}
{\bf (ii)~} N(2,2,3,3;n) = &4\sigma(n) - 8\sigma (n/2) -12 \sigma (n/3) + 16 \sigma (n/4) \\
 & +24\sigma(n/6) - 48\sigma(n/12),
\end{aligned}\\
&\begin{aligned}
{\bf (iii)~} N(1,2,6,6;n) = &\frac{4}{5} \sigma_{(\chi_1,\chi_8)}(n) + \frac{24}{5}\sigma_{(\chi_1,\chi_8)}(n/3) +\frac{2}{5}\sigma_{(\chi_8,\chi_1)}(n) \\
 &-\frac{12}{5}\sigma_{(\chi_8,\chi_1)}(n/3)+\frac{8}{5}b_1(n) - \frac{4}{5}b_2(n),
\end{aligned}\\
&\begin{aligned}
{\bf (iv)~} N(2,2,3,6;n) = &\frac{8}{5} \sigma_{(\chi_1,\chi_8)}(n) - \frac{12}{5}\sigma_{(\chi_1,\chi_8)}(n/3) - \frac{4}{5}\sigma_{(\chi_8,\chi_1)}(n) \\
 &-\frac{6}{5}\sigma_{(\chi_8,\chi_1)}(n/3) - \frac{4}{5}b_2(n),
\end{aligned}\\
&\begin{aligned}
{\bf (v)~}N(1,3,3,3;n) = &  2\sigma_{(\chi_1,\chi_{12})}(n)- \sigma_{(\chi_{12},\chi_1)}(n) +2\sigma_{(\chi_{-3},\chi_{-4})}(n) -\sigma_{(\chi_{-4},\chi_{-3})}(n) ,
\end{aligned}\\
&\begin{aligned}
{\bf (vi)~}N(1,3,6,6;n) = &  \sigma_{(\chi_1,\chi_{12})}(n)- \sigma_{(\chi_{12},\chi_1)}(n/2) +\sigma_{(\chi_{-3},\chi_{-4})}(n)  + \sigma_{(\chi_{-4},\chi_{-3})}(n/2) ,
\end{aligned}\\
&\begin{aligned}
{\bf (vii)~}N(2,3,3,6;n) = &  \sigma_{(\chi_1,\chi_{12})}(n)- \sigma_{(\chi_{12},\chi_1)}(n/2) - \sigma_{(\chi_{-3},\chi_{-4})}(n) -\sigma_{(\chi_{-4},\chi_{-3})}(n/2) ,
\end{aligned}\\
&\begin{aligned}
{\bf (viii)~}
 N(1,3,3,6;n) = & \frac{4}{3}\sigma_{(\chi_1,\chi_{24})}(n) -\frac{1}{3}\sigma_{(\chi_{24},\chi_1)}(n) +\frac{4}{3}\sigma_{(\chi_{-3},\chi_{-8})}(n) -\frac{1}{3}\sigma_{(\chi_{-8},\chi_{-3})}(n) ,
\end{aligned}\\
&\begin{aligned}
{\bf (ix)~}
 N(1,6,6,6;n) = & \frac{2}{3}\sigma_{(\chi_1,\chi_{24})}(n) -\frac{1}{3}\sigma_{(\chi_{24},\chi_1)}(n) +\frac{2}{3}\sigma_{(\chi_{-3},\chi_{-8})}(n)  \\
 &-\frac{1}{3}\sigma_{(\chi_{-8},\chi_{-3})}(n)+\frac{8}{3}c_1(n) +\frac{4}{3}c_2(n),
\end{aligned}\\
&\begin{aligned}
{\bf (x)~}
 N(2,2,2,3;n) = & 2\sigma_{(\chi_1,\chi_{24})}(n) -\frac{1}{3}\sigma_{(\chi_{24},\chi_1)}(n) +\frac{2}{3}\sigma_{(\chi_{-3},\chi_{-8})}(n)  \\
 &-\sigma_{(\chi_{-8},\chi_{-3})}(n) - \frac{4}{3}c_2(n),
\end{aligned}\\
&\begin{aligned}
{\bf (xi)~}
 N(2,3,3,3;n) = & \frac{4}{3}\sigma_{(\chi_1,\chi_{24})}(n) -\frac{1}{3}\sigma_{(\chi_{24},\chi_1)}(n) -\frac{4}{3}\sigma_{(\chi_{-3},\chi_{-8})}(n)  \\
 &+\frac{1}{3}\sigma_{(\chi_{-8},\chi_{-3})}(n) - \frac{8}{3}c_1(n),
\end{aligned}\\
&\begin{aligned}
{\bf (xii)~}
 N(2,3,6,6;n) =
  & \frac{1}{3}\big(2\sigma_{(\chi_1,\chi_{24})}(n) -\sigma_{(\chi_{24},\chi_1)}(n) -2\sigma_{(\chi_{-3},\chi_{-8})}(n)  +\sigma_{(\chi_{-8},\chi_{-3})}(n)\big) .
\end{aligned}
\end{align*}
\end{theorem}

\begin{remark}
{\rm
The formula in Theorem 4.1(i) is the classical result of Jacobi \cite{jacobi, williams-2011}. 
The formulas in Theorems 4.1(xiii)(xiv) and  4.2(viii)(xii)  agree with the formulas given in \cite[p. 1668]{AyseKSW}. 
Note that $A(n)$, $B(n)$, $C(n)$ and $D(n)$ in \cite{AyseKSW} are $\sigma_{(\chi_1,\chi_{24})}(n)$,  $\sigma_{(\chi_{-8},\chi_{-3})}(n)$, $\sigma_{(\chi_{-3},\chi_{-8})}(n)$ and  $\sigma_{(\chi_{24},\chi_1)}(n)$ in this paper, respectively. The formulas in Theorem 4.1(v)(vii) agree 
 with those given in  \cite{williams-2008}.  
The formulas in Theorems 4.1(i)--(iv)  and 4.2(i)(ii)  agree  with those given in  \cite{19quaternary}.
}
 \end{remark}

\section{Eta quotients in  $E_2(\Gamma_0(24),\chi_i)$ for  $ i\in \{1, 8, 12, 24\}$}

Using MAPLE, we  found that there are exactly $819$, $212$, $800$ and $212$ eta quotients in $M_2(\Gamma_0(24),\chi_i)$ for $i\in\{1, 8, 12, 24\}$, respectively. 
Of these, $282$, $8$, $800$ and $8$  eta quotients  are in $E_2(\Gamma_0(24),\chi_i)$ for $i\in\{1, 8, 12, 24\}$, respectively. 
We note that as $\dim(S_2(\Gamma_0(24),\chi_{12}))=0$, we have $M_2(\Gamma_0(24),\chi_{12}) = E_2(\Gamma_0(24),\chi_{12})$.  

Of the eta quotients in $E_2(\Gamma_0(24))$, 250 arise directly from those given in \cite{williams-2012}. 
In Table 5.1 we list the remaining 32 eta quotients and their Fourier series expansions. 

In Table 5.2 we list all  $8$ eta quotients in $E_2(\Gamma_0(24), \chi_{24})$. 
All $8$ eta quotients in $E_2(\Gamma_0(24),\chi_8)$  arise directly from those in $E_2(\Gamma_0(8),\chi_8)$ given in \cite{alacaaygin-2}, 
so we do not list them here.  

All the eta quotients in $E_2(\Gamma_0(12),\chi_{12})$ are given in \cite{alacaaygin-1}. 
There are $395$  eta quotients in $E_2(\Gamma_0(24),\chi_{12})$, which do not  arise directly from those in  $E_2(\Gamma_0(12),\chi_{12})$. 
We list these eta quotients  and  their Fourier series expansions in Table 5.3.

\vspace{1mm}

Theorems 5.1--5.3 follow from Theorem 2.1 directly.

\begin{theorem} 
Let  $f(z)=\ds \prod_{1\leq \delta\mid 24} \eta^{r_{\delta}}(\delta z)\in E_2(\Gamma_0(24))$ be any of the eta quotients with the exponents 
$r_{\delta}$ given on the left hand side of {\rm Table 5.1}. Then we have 
\begin{align*}
f(z)= b_2 L_2(q)+b_3 L_3(q) + b_4 L_4(q) + b_6 L_6(q) +b_8 L_8(q) + b_{12} L_{12}(q) +b_{24} L_{24}(q),
\end{align*}
where the coefficients $b_j$ $(j\in\{2,3,4,6,8,12,24\})$ are given at the right hand side of {\rm Table 5.1}.
\end{theorem}

{\scriptsize
\begin{center}
\begin{longtable}{r r r r r r r r| r r r r r r r }
\caption{ $f(z)= b_2 L_2(q)+b_3 L_3(q) + b_4 L_4(q) + b_6 L_6(q) +b_8 L_8(q) + b_{12} L_{12}(q) +b_{24} L_{24}(q)$. } \\
\hline\hline 
$r_1$ & $r_2$ & $r_3$ & $r_4$ & $r_6$ & $r_{8}$ & $r_{12}$ & $r_{24}$ & $b_2$ & $b_3$ & $b_4$ & $b_6$ & $b_8$ & $b_{12}$ & $b_{24}$  \\ 
\hline
\endfirsthead
\hline
\hline 
 $r_1$ & $r_2$ & $r_3$ & $r_4$ & $r_6$ & $r_{8}$ & $r_{12}$ & $r_{24}$ & $b_2$ & $b_3$ & $b_4$ & $b_6$ & $b_8$ & $b_{12}$ & $b_{24}$  \\ 
\hline
\endhead
\hline 
\endfoot
\hline
\endlastfoot
$-1$ & $ 4$ & $ -1$ & $ -5$ & $ -2$ & $ 2$ & $ 9$ & $ -2$ & $  -1/2$ & $  -1/3$ & $  5/4$ & $  1/3$ & $  -1/2$ & $  -5/12$ & $  1/6$ \\
$-1$ & $ 2$ & $ -1$ & $ 1$ & $ 0$ & $ -2$ & $ 3$ & $ 2$ & $  0$ & $  -1/3$ & $  -1/4$ & $  1/6$ & $  1/2$ & $  1/12$ & $  -1/6$ \\
$-1$ & $ 0$ & $ -1$ & $ 3$ & $ 2$ & $ 2$ & $ 1$ & $ -2$ & $  1/2$ & $  0$ & $  1/4$ & $  0$ & $  -1/2$ & $  -3/4$ & $  3/2$ \\
$-1$ & $ -2$ & $ -1$ & $ 9$ & $ 4$ & $ -2$ & $ -5$ & $ 2$ & $  1$ & $  0$ & $  -5/4$ & $  -3/2$ & $  1/2$ & $  15/4$ & $  -3/2$ \\
$-2$ & $ 5$ & $ 2$ & $ -4$ & $ -5$ & $ 1$ & $ 6$ & $ 1$ & $  0$ & $  -1/3$ & $  -1/2$ & $  5/6$ & $  1/4$ & $  -1/6$ & $  -1/12$ \\
$-2$ & $ 4$ & $ 2$ & $ -1$ & $ -6$ & $ -1$ & $ 9$ & $ -1$ & $  -1/2$ & $  -2/3$ & $  1/2$ & $  11/6$ & $  -1/2$ & $  -5/6$ & $  1/6$ \\
$-2$ & $ 1$ & $ 2$ & $ 4$ & $ -1$ & $ 1$ & $ -2$ & $ 1$ & $  1/2$ & $  0$ & $  -1/2$ & $  0$ & $  1/4$ & $  3/2$ & $  -3/4$ \\
$-2$ & $ 0$ & $ 2$ & $ 7$ & $ -2$ & $ -1$ & $ 1$ & $ -1$ & $  1/2$ & $  0$ & $  1/2$ & $  3/2$ & $  -1/2$ & $  -3/2$ & $  3/2$ \\
$-1$ & $ 1$ & $ -1$ & $ -2$ & $ 7$ & $ 2$ & $ 0$ & $ -2$ & $  1$ & $  1/3$ & $  -1$ & $  -1/3$ & $  0$ & $  -1/3$ & $  4/3$ \\
$-1$ & $ -1$ & $ -1$ & $ 4$ & $ 9$ & $ -2$ & $ -6$ & $ 2$ & $  1$ & $  1/3$ & $  -1$ & $  -5/3$ & $  0$ & $  11/3$ & $  -4/3$ \\
$-2$ & $ 2$ & $ 2$ & $ -1$ & $ 4$ & $ 1$ & $ -3$ & $ 1$ & $  1/2$ & $  1/3$ & $  0$ & $  -5/6$ & $  0$ & $  5/3$ & $  -2/3$ \\
$-2$ & $ 1$ & $ 2$ & $ 2$ & $ 3$ & $ -1$ & $ 0$ & $ -1$ & $  1$ & $  2/3$ & $  0$ & $  -1/3$ & $  0$ & $  -2/3$ & $  4/3$ \\
$2$ & $ -1$ & $ -2$ & $ -2$ & $ 1$ & $ 1$ & $ 4$ & $ 1$ & $  0$ & $  -1/3$ & $  1/2$ & $  1/6$ & $  -1/4$ & $  -1/6$ & $  1/12$ \\
$2$ & $ -2$ & $ -2$ & $ 1$ & $ 0$ & $ -1$ & $ 7$ & $ -1$ & $  -1/2$ & $  2/3$ & $  1/2$ & $  -1/6$ & $  -1/2$ & $  -1/6$ & $  1/6$ \\
$2$ & $ -5$ & $ -2$ & $ 6$ & $ 5$ & $ 1$ & $ -4$ & $ 1$ & $  5/2$ & $  0$ & $  -1/2$ & $  0$ & $  -1/4$ & $  -3/2$ & $  3/4$ \\
$2$ & $ -6$ & $ -2$ & $ 9$ & $ 4$ & $ -1$ & $ -1$ & $ -1$ & $  -11/2$ & $  0$ & $  5/2$ & $  3/2$ & $  -1/2$ & $  -3/2$ & $  3/2$ \\
$1$ & $ 1$ & $ 1$ & $ -4$ & $ -5$ & $ 2$ & $ 10$ & $ -2$ & $  -1/2$ & $  1/3$ & $  5/4$ & $  -2/3$ & $  -1/2$ & $  -1/12$ & $  1/6$ \\
$1$ & $ -1$ & $ 1$ & $ 2$ & $ -3$ & $ -2$ & $ 4$ & $ 2$ & $  0$ & $  -1/3$ & $  1/4$ & $  5/6$ & $  -1/2$ & $  -5/12$ & $  1/6$ \\
$1$ & $ -3$ & $ 1$ & $ 4$ & $ -1$ & $ 2$ & $ 2$ & $ -2$ & $  -5/2$ & $  0$ & $  5/4$ & $  0$ & $  -1/2$ & $  -3/4$ & $  3/2$ \\
$1$ & $ -5$ & $ 1$ & $ 10$ & $ 1$ & $ -2$ & $ -4$ & $ 2$ & $  2$ & $  0$ & $  1/4$ & $  3/2$ & $  -1/2$ & $  -15/4$ & $  3/2$ \\
$2$ & $ -4$ & $ -2$ & $ 1$ & $ 10$ & $ 1$ & $ -5$ & $ 1$ & $  5/2$ & $  1/3$ & $  -1$ & $  -1/6$ & $  0$ & $  -4/3$ & $  2/3$ \\
$2$ & $ -5$ & $ -2$ & $ 4$ & $ 9$ & $ -1$ & $ -2$ & $ -1$ & $  -5$ & $  -2/3$ & $  2$ & $  5/3$ & $  0$ & $  -4/3$ & $  4/3$ \\
$1$ & $ -2$ & $ 1$ & $ -1$ & $ 4$ & $ 2$ & $ 1$ & $ -2$ & $  -2$ & $  -1/3$ & $  0$ & $  2/3$ & $  0$ & $  -2/3$ & $  4/3$ \\
$1$ & $ -4$ & $ 1$ & $ 5$ & $ 6$ & $ -2$ & $ -5$ & $ 2$ & $  2$ & $  1/3$ & $  0$ & $  2/3$ & $  0$ & $  -10/3$ & $  4/3$ \\
$-1$ & $ 9$ & $ -1$ & $ -6$ & $ -1$ & $ 2$ & $ 4$ & $ -2$ & $  5$ & $  3$ & $  -11$ & $  -3$ & $  4$ & $  3$ & $  0$ \\
$-1$ & $ 7$ & $ -1$ & $ 0$ & $ 1$ & $ -2$ & $ -2$ & $ 2$ & $  1$ & $  3$ & $  1$ & $  -3$ & $  -4$ & $  3$ & $  0$ \\
$-2$ & $ 10$ & $ 2$ & $ -5$ & $ -4$ & $ 1$ & $ 1$ & $ 1$ & $  1/2$ & $  3$ & $  4$ & $  -15/2$ & $  -2$ & $  3$ & $  0$ \\
$-2$ & $ 9$ & $ 2$ & $ -2$ & $ -5$ & $ -1$ & $ 4$ & $ -1$ & $  5$ & $  6$ & $  -4$ & $  -15$ & $  4$ & $  6$ & $  0$ \\
$2$ & $ 4$ & $ -2$ & $ -3$ & $ 2$ & $ 1$ & $ -1$ & $ 1$ & $  5/2$ & $  3$ & $  -5$ & $  -3/2$ & $  2$ & $  0$ & $  0$ \\
$2$ & $ 3$ & $ -2$ & $ 0$ & $ 1$ & $ -1$ & $ 2$ & $ -1$ & $  -1$ & $  -6$ & $  -2$ & $  3$ & $  4$ & $  0$ & $  0$ \\
$1$ & $ 6$ & $ 1$ & $ -5$ & $ -4$ & $ 2$ & $ 5$ & $ -2$ & $  2$ & $  -3$ & $  -10$ & $  6$ & $  4$ & $  0$ & $  0$ \\
$1$ & $ 4$ & $ 1$ & $ 1$ & $ -2$ & $ -2$ & $ -1$ & $ 2$ & $  2$ & $  3$ & $  -2$ & $  -6$ & $  4$ & $  0$ & $  0$ 
\end{longtable}
\end{center}
}

\begin{theorem} 
Let  $f(z)=\ds \prod_{1\leq \delta\mid 24} \eta^{r_{\delta}}(\delta z)\in E_2(\Gamma_0(24),\chi_{24})$ be any of the eta quotients with the exponents 
$r_{\delta}$ given on the left hand side of {\em Table 5.2}, where $\chi_{24}(n)$ is given by {\em (1.1)}. Then we have 
\begin{align*}
f(z)= b_1 E_{24, 1} (z)+b_2 E_{1,24} (z)+b_3 E_{ -8,-3} (z)+ b_4 E_{ -3,-8} (z),
\end{align*}
where the coefficients $b_1$, $b_{2}$, $b_3$, $b_4$ are given at the right hand side of {\em Table 5.2}.
\end{theorem}

{\small 
\begin{center}
\begin{longtable}{r r r r r r r r| r r r r}
\caption{\scriptsize 
$f(z)=b_1 E_{24, 1} (z)+b_2 E_{1,24} (z)+b_3 E_{ -8,-3} (z)+ b_4 E_{ -3,-8} (z)$. } \\
\hline\hline 
$r_1$ & $r_2$ & $r_3$ & $r_4$ & $r_6$ & $r_{8}$ & $r_{12}$ & $r_{24}$ & $b_1$ & $b_2$ & $b_{3}$ & $b_4$   \\ 
\hline
\endfirsthead
\hline
\hline 
 $r_1$ & $r_2$ & $r_3$ & $r_4$ & $r_6$ & $r_{8}$ & $r_{12}$ & $r_{24}$ & $b_1$ & $b_2$ & $b_{3}$ & $b_4$   \\ 
\hline
\endhead
\hline 
\endfoot
\hline
\endlastfoot
$0$ & $ -2$ & $ -2$ & $ 5$ & $ 1$ & $ -2$ & $ 8$ & $ -4$ & $ -1/3$ & $ 2/3$ & $ 1/3$ & $ -2/3 $ \\
 $-1$ & $ -1$ & $ 1$ & $ 6$ & $ -2$ & $ -3$ & $ 5$ & $ -1$ & $ 0$ & $ 2/3$ & $ 1/3$ & $ 0 $ \\
 $-2$ & $ 1$ & $ 0$ & $ 8$ & $ -2$ & $ -4$ & $ 5$ & $ -2$ & $ -1/3$ & $ 2$ & $ 1$ & $ -2/3 $ \\
 $-2$ & $ 5$ & $ -4$ & $ -2$ & $ 8$ & $ 0$ & $ 1$ & $ -2$ & $ -1/3$ & $ 4/3$ & $ -1/3$ & $ 4/3 $ \\
 $-3$ & $ 6$ & $ -1$ & $ -1$ & $ 5$ & $ -1$ & $ -2$ & $ 1$ & $ 0$ & $ 4/3$ & $ -1/3$ & $ 0 $ \\
 $-4$ & $ 8$ & $ -2$ & $ 1$ & $ 5$ & $ -2$ & $ -2$ & $ 0$ & $ -1/3$ & $ 4$ & $ -1$ & $ 4/3 $ \\
 $1$ & $ -2$ & $ -1$ & $ 5$ & $ -1$ & $ -1$ & $ 6$ & $ -3$ & $ -1/3$ & $ 0$ & $ 0$ & $ -2/3 $ \\
 $-1$ & $ 5$ & $ -3$ & $ -2$ & $ 6$ & $ 1$ & $ -1$ & $ -1$ & $ -1/3$ & $ 0$ & $ 0$ & $ 4/3$
\end{longtable}
\end{center}
}

\begin{theorem} 
Let  $f(z)=\ds \prod_{1\leq \delta\mid 24} \eta^{r_{\delta}}(\delta z)\in E_2(\Gamma_0(24),\chi_{12})$ be any of the eta quotients with the exponents 
$r_{\delta}$ given on the left hand side of {\em Table 5.3}, where $\chi_{12}(n)$ is given by {\em (1.1)}. Then we have 
\begin{align*}
f(z)=& b_1 E_{12,1}(z) +b_2 E_{12,1} (2z)+ b_3E_{1,12}(z) +b_4 E_{1,12} (2z) +b_5 E_{-4,-3} (z) \\
&+b_6 E_{-4,-3} (2z) +b_7 E_{-3,-4} (z) +b_8 E_{-3,-4} (2z),
\end{align*}
where the coefficients $b_i$ $(1\leq i\leq 8)$ are given at the right hand side of {\em Table 5.3}.
\end{theorem}

{\scriptsize
\begin{landscape}
\begin{center}
\def\arraystretch{0.6}
\begin{longtable}{r r r r r r r r | r r r r r r r r}
\caption{\scriptsize{$f(z)=b_1 E_{12,1}(z) +b_2 E_{12,1} (2z)+ b_3E_{1,12}(z) +b_4 E_{1,12} (2z) +b_5 E_{-4,-3} (z) 
+b_6 E_{-4,-3} (2z) +b_7 E_{-3,-4} (z) +b_8 E_{-3,-4} (2z)$ }} \\
\hline\hline 
$r_1$ & $r_2$ & $r_3$ & $r_{4}$ & $r_6$ & $r_{8}$ & $r_{12}$ & $r_{24}$ & $b_1$ & $b_2$ & $b_3$ & $b_{4}$ & $b_5$ & $b_6$ & $b_7$ & $b_8$\\ 
\hline
\endfirsthead
\hline
\hline 
$r_1$ & $r_2$ & $r_3$ & $r_{4}$ & $r_6$ & $r_{8}$ & $r_{12}$ & $r_{24}$ & $b_1$ & $b_2$ & $b_3$ & $b_{4}$ & $b_5$ & $b_6$ & $b_7$ & $b_8$\\  
\hline
\endhead
\hline 
\endfoot
\hline
\endlastfoot
$0$ & $ 0$ & $ 0$ & $ 1$ & $ 0$ & $ -2$ & $ -5$ & $ 10$ & $  -1/16$ & $  1/16$ & $  1/48$ & $  -1/24$ & $  -1/48$ & $  -1/48$ & $  1/16$ & $  1/8 $\\
$0$ & $ 1$ & $ 0$ & $ -2$ & $ -3$ & $ 0$ & $ 4$ & $ 4$ & $  0$ & $  0$ & $  1/24$ & $  -1/12$ & $  1/12$ & $  1/12$ & $  -1/8$ & $  -1/4 $\\
$0$ & $ 2$ & $ 0$ & $ -5$ & $ -6$ & $ 2$ & $ 13$ & $ -2$ & $  0$ & $  0$ & $  1/12$ & $  -1/6$ & $  -1/3$ & $  -1/3$ & $  1/4$ & $  1/2 $\\
$0$ & $ 3$ & $ 0$ & $ -8$ & $ -9$ & $ 4$ & $ 22$ & $ -8$ & $  0$ & $  0$ & $  1/6$ & $  -1/3$ & $  4/3$ & $  4/3$ & $  -1/2$ & $  -1 $\\
$0$ & $ 0$ & $ 0$ & $ -1$ & $ 0$ & $ 2$ & $ -3$ & $ 6$ & $  -1/16$ & $  1/16$ & $  1/16$ & $  -1/8$ & $  1/16$ & $  1/16$ & $  -1/16$ & $  -1/8 $\\
$0$ & $ 1$ & $ 0$ & $ -4$ & $ -3$ & $ 4$ & $ 6$ & $ 0$ & $  0$ & $  0$ & $  1/8$ & $  -1/4$ & $  -1/4$ & $  -1/4$ & $  1/8$ & $  1/4 $\\
$0$ & $ 2$ & $ 0$ & $ -7$ & $ -6$ & $ 6$ & $ 15$ & $ -6$ & $  0$ & $  0$ & $  1/4$ & $  -1/2$ & $  1$ & $  1$ & $  -1/4$ & $  -1/2 $\\
$0$ & $ 0$ & $ 0$ & $ -3$ & $ 0$ & $ 6$ & $ -1$ & $ 2$ & $  -1/16$ & $  1/16$ & $  3/16$ & $  -3/8$ & $  -3/16$ & $  -3/16$ & $  1/16$ & $  1/8 $\\
\end{longtable}
\end{center}
\end{landscape} }

\section*{Acknowledgments} 
The authors are grateful to Professor Emeritus Kenneth S. Williams for helpful discussions throughout the course of this research. 
The research of the first two authors was supported 
by Discovery Grants from the Natural Sciences and Engineering Research Council of Canada (RGPIN-418029-2013 and RGPIN-2015-05208). 
Zafer Selcuk Aygin's studies are supported by Turkish Ministry of Education.

\vspace{2mm}
\noindent
Centre for Research in Algebra and Number Theory \\
School of Mathematics and Statistics \\
Carleton University\\
Ottawa, Ontario, K1S 5B6, Canada \\

\noindent
AyseAlaca@cunet.carleton.ca\\
SabanAlaca@cunet.carleton.ca\\
ZaferAygin@cmail.carleton.ca

\end{document}